\documentclass[11pt,reqno]{amsart}
\usepackage[margin=0.8in]{geometry}
\usepackage{amsmath,amssymb,amsthm}
\usepackage{microtype}
\usepackage{hyperref}
\hypersetup{hidelinks}
\newtheorem{theorem}{Theorem}[section]
\newtheorem{lemma}[theorem]{Lemma}
\newtheorem{proposition}[theorem]{Proposition}
\newtheorem{corollary}[theorem]{Corollary}
\newtheorem{introtheorem}{Theorem}

\theoremstyle{definition}

\theoremstyle{remark}
\newtheorem{remark}[theorem]{Remark}
\numberwithin{equation}{section}

\title{On Intersections of System Normalizers}

\author{Ting Gong}
\address{Department of Mathematics, University of Washington, Seattle, WA 98195}
\email{tgong2@uw.edu}

\author{Yuan Lu}
\address{Department of Mathematics, ETH Zürich, 8092 Zürich, Switzerland}
\email{yuan.lu@math.ethz.ch}

\author{Yong Yang}
\address{Department of Mathematics, Texas State University, San Marcos, TX 78666}
\email{yang@txstate.edu}

\author{Michael Ruofan Zeng}
\address{Department of Mathematics, University of Washington, Seattle, WA 98195}
\email{zengrf@uw.edu}
\date{}

\begin{document}
\begin{abstract}
We give a negative answer to Problem 17.39 of the Kourovka Notebook. For every odd prime power $q$ and every $m\geq1$, we construct a finite metabelian group $G_{m,q}$ with $Z_\infty(G_{m,q})=1$ and $|\Phi(G_{m,q})|=q$ for which the least number of system normalizers with trivial intersection is $m+1$. Thus the number is unbounded even when the order of the Frattini subgroup is fixed. 
\end{abstract}
\maketitle

\section{Introduction}\label{sec:introduction}

Let $G$ be a finite solvable group. Write $\Phi(G)$ for the Frattini subgroup of $G$, the intersection of its maximal subgroups, and write $Z_\infty(G)$ for the hypercenter of $G$. The intersection of all system normalizers of $G$ is $Z_\infty(G)$; see, for example, \cite{DoerkHawkes}. Problem 17.39 of the Kourovka Notebook \cite{Kourovka20} asks whether there exists a positive integer $n$, independent of $G$, such that $Z_\infty(G)$ is the intersection of $n$ system normalizers. Ballester-Bolinches, Cossey, Kamornikov, and Meng proved that if $\Phi(G)=1$, then four conjugate system normalizers suffice; under additional hypotheses, three suffice \cite{BCKM}. We show that no uniform bound exists in general, even for metabelian groups. In the groups used below all system normalizers are conjugate, so the same examples also rule out a bound when one restricts to conjugate system normalizers.

\begin{introtheorem}\label{thm:A}
For every odd prime power $q$ and every integer $m\geq1$, there exists a finite metabelian group $G_{m,q}$ such that
\[
 |G_{m,q}|=2q^{2m+1},\qquad
 Z_\infty(G_{m,q})=1,\qquad |\Phi(G_{m,q})|=q,
\]
and the least number of system normalizers whose intersection is the hypercenter is $m+1$. For the Frattini quotient $G_{m,q}/\Phi(G_{m,q})$, the least number is two.
\end{introtheorem}

\begin{corollary}\label{cor:unbounded}
For every odd prime power $q$, there is no bound on the number of system normalizers needed to express the hypercenter as their intersection among finite metabelian groups $G$ with $|\Phi(G)|=q$. Consequently, there is no absolute bound for finite solvable groups, and no such bound can depend only on $|\Phi(G)|$.
\end{corollary}

The groups in Theorem~\ref{thm:A} form a two-parameter family. Fixing $q$ keeps $|\Phi(G_{m,q})|=q$ while the required number of system normalizers tends to infinity with $m$. We construct $G_{m,q}$ as a semidirect product $\mathsf V\rtimes H$. The subgroup $H$ is a system normalizer and every system normalizer is conjugate to $H$. We determine the intersections of these conjugates, then compute $\Phi(G_{m,q})$ and the system normalizers of the Frattini quotient.

\section{Background}\label{sec:background}

For a finite group \(G\) and a subset \(\pi\) of the prime divisors of \(|G|\), a Hall \(\pi\)-subgroup has order divisible only by primes in \(\pi\). Its index is divisible only by primes outside \(\pi\). A Hall system \(\mathcal H\) consists of one Hall \(\pi\)-subgroup for each such subset, with every two members permuting as subgroups. For a subgroup \(K\), write \(\operatorname{N}_G(K)=\{g\in G:g^{-1}Kg=K\}\) for its normalizer. The system normalizer of \(\mathcal H\) is
\[
 D(\mathcal H)=\bigcap_{K\in\mathcal H}\operatorname{N}_G(K).
\]

\begin{lemma}\label{lem:hall}
Let $p$ be an odd prime. If \(r\geq1\) and \(G\) is a group of order \(2p^r\) with a normal subgroup \(P\) of order \(p^r\), then its system normalizers are \(\operatorname{N}_G(T)\), where \(T\) ranges over the subgroups of order two in \(G\).
\end{lemma}

\begin{proof}
The prime divisors of $|G|$ are $2$ and $p$, so a Hall system has one Hall subgroup for each of the four subsets
\[
 \varnothing,\qquad \{2\},\qquad \{p\},\qquad \{2,p\}.
\]
The Hall $\varnothing$-subgroup is $1$ and the Hall $\{2,p\}$-subgroup is $G$. Since $P$ is normal of index two, it is the unique Hall $\{p\}$-subgroup. A Hall $\{2\}$-subgroup has order two, so it is a subgroup $T$ of order two. Thus every Hall system has the form
\[
 \{1,T,P,G\}.
\]
Conversely, for every subgroup $T$ of order two, these four subgroups form a Hall system. Since $P\lhd G$ and $P\cap T=1$, we have $PT=TP=G$. The remaining permutability relations are immediate. Hence
\[
 D(\mathcal H)=N_G(1)\cap N_G(T)\cap N_G(P)\cap N_G(G)=N_G(T),
\]
because $N_G(P)=G$. This proves the assertion.
\end{proof}

The upper central series is defined by \(Z_0(G)=1\) and
\[
 Z_{i+1}(G)/Z_i(G)=Z\bigl(G/Z_i(G)\bigr)
 \qquad(i\geq0),
\]
where \(Z(G)\) is the center. For finite \(G\), the ascending series stabilizes at \(Z_\infty(G)\). If \(Z(G)=1\), then induction gives \(Z_i(G)=1\) for every \(i\), so \(Z_\infty(G)=1\). If \(G/Z(G)\) is centerless, then the same recursion gives \(Z_\infty(G)=Z(G)\).

\section{The groups}\label{sec:main}

\subsection{System normalizers}

Fix an odd prime power \(q=p^f\) and an integer \(m\geq1\), and put \(\mathsf X=\mathbb F_q^m\) and \(\mathsf V=\mathsf X\oplus\mathbb F_q\). For \(\mathbf b\in\mathsf X\), define
\begin{equation}\label{eq:shear}
 \begin{aligned}[b]
 u_{\mathbf b}(\mathbf x,z)&=(\mathbf x,z+\mathbf b\cdot\mathbf x),\\
 \mathbf b\cdot\mathbf x&=\sum_{j=1}^m b_jx_j
 \end{aligned}.
\end{equation}
The standard basis vectors of \(\mathsf X\) are denoted by \(\mathbf e_1,\ldots,\mathbf e_m\). Composition gives \(u_{\mathbf b}u_{\mathbf c}=u_{\mathbf b+\mathbf c}\). Evaluation at the vectors \((\mathbf e_j,0)\) shows that \(u_{\mathbf b}=1\) only when \(\mathbf b=\mathbf0\). Thus \(U=\{u_{\mathbf b}:\mathbf b\in\mathsf X\}\) is elementary abelian of order \(q^m\). Let \(t\) be scalar inversion on \(\mathsf V\), so \(t(\mathbf v)=-\mathbf v\). This transformation centralizes \(U\), has order two, and does not belong to \(U\). We define
\begin{equation}\label{eq:group}
 H=U\times\langle t\rangle,\qquad
 G_{m,q}=\mathsf V\rtimes H,\qquad P=\mathsf V\rtimes U.
\end{equation}
The group operation and the action on $\mathsf V$ are
\begin{equation}\label{eq:affine}
 \begin{aligned}[b]
 (\mathbf v,h)(\mathbf w,h')&=(\mathbf v+h(\mathbf w),hh'),\\
 (\mathbf v,h)\cdot\mathbf a&=\mathbf v+h(\mathbf a)
 \end{aligned}.
\end{equation}
The translation by \(\mathbf v\) is \(\tau_{\mathbf v}=(\mathbf v,1)\). We identify \(H\) with the subgroup of elements having translation part zero. For $s\in\{1,-1\}$, write $sI_{\mathsf V}$ for scalar multiplication by $s$ on $\mathsf V$; thus every element of $H$ has a unique form $sI_{\mathsf V}u_{\mathbf b}$. Both $\mathsf V$ and $H$ are abelian, and $P$ is the inverse image of $U$ under the projection to $H$.

\begin{lemma}\label{lem:affine}
For every odd prime power \(q\) and every integer \(m\geq1\), the group \(G_{m,q}\) defined by (\ref{eq:shear})--(\ref{eq:affine}) is metabelian of order \(2q^{2m+1}\), has \(Z_\infty(G_{m,q})=1\), and has exactly \(q^{m+1}\) system normalizers. They are the conjugates of \(H\).
\end{lemma}

\begin{proof}
The quotient \(G_{m,q}/\mathsf V\cong H\) is abelian, so the commutator subgroup \(G_{m,q}'\) is contained in the abelian group \(\mathsf V\). Thus \(G_{m,q}\) is metabelian. Its order is \(|\mathsf V||H|=2q^{2m+1}\), and \(P\) is a normal Sylow \(p\)-subgroup.

The normalizer of \(T=\langle(\mathbf0,t)\rangle\) equals its centralizer, since an order-two group has only the identity automorphism. For \((\mathbf v,h)\in G_{m,q}\), the two products are
\[
 \begin{aligned}[b]
 (\mathbf v,h)(\mathbf0,t)&=(\mathbf v,ht),\\
 (\mathbf0,t)(\mathbf v,h)&=(-\mathbf v,th)
 \end{aligned}.
\]
Since \(t\) centralizes \(H\), equality of the products is equivalent to \(\mathbf v=-\mathbf v\). Since $q$ is odd, the translation part must be zero. Consequently, the normalizer of \(T\) is exactly \(H\).

Every involution in \(G_{m,q}\) has linear part \(t\). Indeed, its image in $H$ has order dividing two, and the only elements of $H$ with this property are $1$ and $t$. An element with linear part $1$ is a translation and has order $1$ or $p$, so a nonidentity involution cannot have linear part $1$. Conversely,
\[
 (\mathbf v,t)^2=(\mathbf v+t(\mathbf v),1)=1,
\]
so every $(\mathbf v,t)$ is an involution, and
\begin{equation}\label{eq:involution}
 \tau_{\mathbf w}(\mathbf0,t)\tau_{\mathbf w}^{-1}
       =(2\mathbf w,t).
\end{equation}
Since $q$ is odd, multiplication by two is a bijection on \(\mathsf V\). Hence (\ref{eq:involution}) makes every subgroup of order two a conjugate of \(T\) by a translation. By Lemma~\ref{lem:hall}, every system normalizer is therefore a conjugate of \(H=N_{G_{m,q}}(T)\). Since $H$ fixes $\mathbf0$ and $t\in H$ fixes no nonzero vector, the common fixed-point set of $H$ on $\mathsf V$ is $\{\mathbf0\}$. Thus $\tau_{\mathbf a}H\tau_{\mathbf a}^{-1}$ fixes exactly $\mathbf a$. Two such conjugates are equal only when their fixed points are equal. Hence $H$ has exactly $|\mathsf V|=q^{m+1}$ conjugates, and these are all the system normalizers of $G_{m,q}$.

Finally, suppose that \((\mathbf v,h)\in Z(G_{m,q})\). For every $\mathbf w\in\mathsf V$,
\[
 (\mathbf v,h)(\mathbf w,1)=(\mathbf v+h(\mathbf w),h),\qquad
 (\mathbf w,1)(\mathbf v,h)=(\mathbf w+\mathbf v,h).
\]
Thus $h(\mathbf w)=\mathbf w$ for every $\mathbf w$. The action of $H$ on $\mathsf V$ is faithful: if $sI_{\mathsf V}u_{\mathbf b}$ acts trivially, its first coordinate gives $s=1$, and evaluation at $(\mathbf e_i,0)$ gives $b_i=0$ for every $i$. Hence $h=1$. Commuting with \((\mathbf0,t)\) then gives \(\mathbf v=-\mathbf v\), so \(\mathbf v=\mathbf0\). Therefore $Z(G_{m,q})=1$, and hence \(Z_\infty(G_{m,q})=1\).
\end{proof}

\begin{proposition}\label{prop:minimum}
For every odd prime power \(q\) and every integer \(m\geq1\), the least number of conjugates of \(H\) with trivial intersection is \(m+1\). Hence the least number of system normalizers of \(G_{m,q}\) with trivial intersection is \(m+1\).
\end{proposition}

\begin{proof}
Let \(1\leq k\leq m\), and let \(D_1,\ldots,D_k\) be system normalizers. By Lemma~\ref{lem:affine}, there are vectors \(\mathbf a_i\in\mathsf V\) such that
\[
 D_i=\tau_{\mathbf a_i}H\tau_{\mathbf a_i}^{-1}
 \qquad(1\leq i\leq k).
\]
Conjugating the intersection by \(\tau_{-\mathbf a_1}\), we may assume \(D_1=H\) and \(\mathbf a_1=\mathbf0\). Write \(\mathbf a_i=(\mathbf x_i,z_i)\) for \(2\leq i\leq k\).

For \(u_{\mathbf b}\in U\leq H\), we have \(u_{\mathbf b}\in D_i\) if and only if \(u_{\mathbf b}\) fixes \(\mathbf a_i\). Hence \(u_{\mathbf b}\in D_1\cap\cdots\cap D_k\) if and only if
\begin{equation}\label{eq:fixed}
 \mathbf b\cdot\mathbf x_i=0\qquad(2\leq i\leq k).
\end{equation}
These equations define the kernel of
\[
 \begin{aligned}[b]
 \mathsf X&\longrightarrow\mathbb F_q^{k-1},\\
 \mathbf b&\longmapsto
 (\mathbf b\cdot\mathbf x_2,\ldots,\mathbf b\cdot\mathbf x_k).
 \end{aligned}
\]
The kernel has dimension at least \(m-k+1>0\). Choose \(0\ne\mathbf b\) in the kernel. Then \(u_{\mathbf b}\ne1\) and \(u_{\mathbf b}\in D_1\cap\cdots\cap D_k\). Thus no \(m\) system normalizers have trivial intersection.

For the upper bound, put \(\mathbf a_i=(\mathbf e_i,0)\) and consider the \(m+1\) conjugates
\[
 H,\quad \tau_{\mathbf a_1}H\tau_{\mathbf a_1}^{-1},\ldots,
 \tau_{\mathbf a_m}H\tau_{\mathbf a_m}^{-1}.
\]
An element in their intersection lies in \(H\), so it has the form \(sI_{\mathsf V}u_{\mathbf b}\), where \(s\in\{1,-1\}\), and it fixes every \(\mathbf a_i\). Now
\[
 sI_{\mathsf V}u_{\mathbf b}(\mathbf e_i,0)=(s\mathbf e_i,sb_i).
\]
Fixing \(\mathbf e_1,0\) gives \(s=1\), and then fixing all \(\mathbf e_i,0\) gives \(b_i=0\) for every \(i\). Hence these \(m+1\) conjugates of \(H\) have trivial intersection.
\end{proof}

\subsection{The Frattini subgroup and the quotient}

Define the translation subgroup \(L=\{\tau_{(\mathbf0,z)}:z\in\mathbb F_q\}\). Each $u_{\mathbf b}$ fixes every vector $(\mathbf0,z)$, while $t$ sends $(\mathbf0,z)$ to $(\mathbf0,-z)$. Hence $H$ normalizes $L$. Translations centralize $L$, so $L\lhd G_{m,q}$. Moreover, $U$ centralizes $L$, and therefore $L\leq Z(P)$ for $P=\mathsf V\rtimes U$. In particular, $|L|=q$.

\begin{lemma}\label{lem:commutator}
For the group \(G_{m,q}\) in (\ref{eq:group}), with \(m\geq1\), the subgroup \(L=\{\tau_{(\mathbf0,z)}:z\in\mathbb F_q\}\) satisfies \(P'=L\leq\Phi(G_{m,q})\), where \(P=\mathsf V\rtimes U\).
\end{lemma}

\begin{proof}
The action gives
\begin{equation}\label{eq:commutator}
 u_{\mathbf b}\tau_{(\mathbf x,z)}u_{\mathbf b}^{-1}
       \tau_{(\mathbf x,z)}^{-1}
   =\tau_{(\mathbf0,\mathbf b\cdot\mathbf x)}.
\end{equation}
Conjugating a translation by a linear transformation replaces its vector by its image. Subtracting \((\mathbf x,z)\) from its image under \(u_{\mathbf b}\) leaves \((\mathbf0,\mathbf b\cdot\mathbf x)\). The dot products exhaust \(\mathbb F_q\), since \(m\geq1\), so \(L\leq P'\). Equation~(\ref{eq:commutator}) shows that $[\mathsf V,U]\leq L$. Since both $\mathsf V$ and $U$ are abelian, it follows that $P/L$ is abelian, and hence $P'\leq L$.

Suppose, to the contrary, that a maximal subgroup $M$ of $G_{m,q}$ does not contain $L$. Since $L\lhd G_{m,q}$, maximality gives $G_{m,q}=ML$. Intersecting with the normal subgroup $P$ and using $L\leq P$ gives
\[
 P=P\cap ML=(P\cap M)L.
\]
Indeed, if $p=ml\in P$ with $m\in M$ and $l\in L$, then $m=pl^{-1}\in P\cap M$. Since $L\leq Z(P)$ and $P=(P\cap M)L$, we have
\[
 L=P'=((P\cap M)L)'=(P\cap M)'\leq M,
\]
a contradiction. Every maximal subgroup contains \(L\), and therefore \(L\leq\Phi(G_{m,q})\).
\end{proof}

Put
\begin{equation}\label{eq:quotient}
 Q_{m,q}=G_{m,q}/L,\qquad
 B_{m,q}=\mathsf X\rtimes\langle-I_{\mathsf X}\rangle.
\end{equation}
Equation~\eqref{eq:commutator} gives $[\mathsf V,U]=L$. Hence the image of $U$ centralizes $\mathsf V/L$ in $Q_{m,q}$. It also centralizes the image of $t$. Thus the image of $U$ is central in $Q_{m,q}$. On the other hand,
\[
 (\mathsf V/L)\rtimes\langle tL\rangle\cong B_{m,q},
\]
since $\mathsf V/L\cong\mathsf X$ and $t$ acts by inversion. This subgroup and the image of $U$ have trivial intersection and generate $Q_{m,q}$. Therefore
\[
 Q_{m,q}\cong B_{m,q}\times U.
\]

\begin{proposition}\label{prop:frattini}
For every odd prime power \(q\) and every integer \(m\geq1\), the group \(G_{m,q}\) in (\ref{eq:group}) satisfies \(\Phi(G_{m,q})=L=\{\tau_{(\mathbf0,z)}:z\in\mathbb F_q\}\), and the quotient \(Q_{m,q}=G_{m,q}/L\) satisfies \(\Phi(Q_{m,q})=1\).
\end{proposition}

\begin{proof}
Under the isomorphism $Q_{m,q}\cong B_{m,q}\times U$, the subgroup $\mathsf X\times U$ has index two and is therefore maximal. Let $\alpha:\mathsf X\to\mathbb F_p$ be a nonzero linear functional. Since $\ker\alpha$ is invariant under $-I_{\mathsf X}$, the subgroups
\begin{equation}\label{eq:maximal}
 \begin{aligned}
 M_\alpha&=(\ker\alpha\rtimes\langle-I_{\mathsf X}\rangle)\times U,\\
 N_\alpha&=B_{m,q}\times\{u_{\mathbf b}:\alpha(\mathbf b)=0\}
 \end{aligned}
\end{equation}
are well defined. Here $\alpha$ is regarded as an $\mathbb F_p$-linear functional on the additive group of $\mathsf X$. Each subgroup has index $p$ in $Q_{m,q}$, and hence is maximal.

Choose an $\mathbb F_p$-basis of $\mathsf X$ and let $\alpha_1,\ldots,\alpha_{fm}$ be the corresponding coordinate functionals. If an element lies in $\mathsf X\times U$, then its scalar component is $1$. If it also lies in every $M_{\alpha_i}$, then its translation component belongs to $\bigcap_i\ker\alpha_i=0$. If it also lies in every $N_{\alpha_i}$, then its shear parameter belongs to $\bigcap_i\ker\alpha_i=0$. Hence
\[
 (\mathsf X\times U)\cap\bigcap_{i=1}^{fm} M_{\alpha_i}
 \cap\bigcap_{i=1}^{fm} N_{\alpha_i}=1.
\]
Since the Frattini subgroup is contained in every maximal subgroup, this gives $\Phi(Q_{m,q})=1$.

The inverse images in $G_{m,q}$ of the maximal subgroups just used are maximal and have intersection exactly $L$, because their intersection in $Q_{m,q}$ is trivial. Hence $\Phi(G_{m,q})\leq L$. Lemma~\ref{lem:commutator} gives the reverse inclusion, and therefore $\Phi(G_{m,q})=L$.
\end{proof}

\begin{proposition}\label{prop:quotient}
For every odd prime power \(q\) and every integer \(m\geq1\), the quotient \(Q_{m,q}\cong B_{m,q}\times U\) in (\ref{eq:quotient}) has hypercenter \(\{1\}\times U\) and has exactly \(q^m\) system normalizers \(J\times U\), with \(J\) ranging over the order-two subgroups of \(B_{m,q}\), for which two is the least number whose intersection equals the hypercenter.
\end{proposition}

\begin{proof}
In $Q_{m,q}\cong B_{m,q}\times U$, the factor $U$ is central. We first show that $B_{m,q}$ is centerless. If $(\mathbf x,sI_{\mathsf X})\in Z(B_{m,q})$, then commuting with all translations forces $s=1$. Commuting with $(\mathbf0,-I_{\mathsf X})$ then gives $\mathbf x=-\mathbf x$, and hence $\mathbf x=\mathbf0$. Thus $Z(B_{m,q})=1$, and therefore
\[
 Z(Q_{m,q})=\{1\}\times U,\qquad Q_{m,q}/Z(Q_{m,q})\cong B_{m,q}.
\]
Since the quotient by $Z(Q_{m,q})$ is centerless, the upper central series has no term strictly above the center. Hence
\[
 Z_\infty(Q_{m,q})=Z(Q_{m,q})=\{1\}\times U.
\]

Every involution in $Q_{m,q}$ has trivial $U$-component, because $U$ has odd order. The involutions of $B_{m,q}$ are the elements $(\mathbf x,-I_{\mathsf X})$, with $\mathbf x\in\mathsf X$. Translation conjugation sends $(\mathbf0,-I_{\mathsf X})$ to $(2\mathbf x,-I_{\mathsf X})$, so these $q^m$ involutions are conjugate and generate $q^m$ distinct order-two subgroups.

Let $J_0=\langle(\mathbf0,-I_{\mathsf X})\rangle$. Since $\operatorname{Aut}(J_0)=1$, its normalizer equals its centralizer. If $(\mathbf x,sI_{\mathsf X})\in B_{m,q}$, then
\[
 (\mathbf x,sI_{\mathsf X})(\mathbf0,-I_{\mathsf X})
   =(\mathbf x,-sI_{\mathsf X}),\qquad
 (\mathbf0,-I_{\mathsf X})(\mathbf x,sI_{\mathsf X})
   =(-\mathbf x,-sI_{\mathsf X}).
\]
Thus it centralizes $(\mathbf0,-I_{\mathsf X})$ only when $\mathbf x=\mathbf0$. Hence $N_{B_{m,q}}(J_0)=J_0$, and by conjugacy $N_{B_{m,q}}(J)=J$ for every order-two subgroup $J$ of $B_{m,q}$.

Now $|Q_{m,q}|=2q^{2m}$ and $\mathsf X\times U$ is a normal subgroup of order $q^{2m}=p^{2fm}$, hence the normal Sylow $p$-subgroup of $Q_{m,q}$. Lemma~\ref{lem:hall} therefore applies. The order-two subgroups of $Q_{m,q}$ are $J\times1$, and
\[
 N_{Q_{m,q}}(J\times1)=N_{B_{m,q}}(J)\times U=J\times U.
\]
Thus the system normalizers of $Q_{m,q}$ are exactly the $q^m$ subgroups $J\times U$. Two distinct ones intersect in $\{1\}\times U=Z_\infty(Q_{m,q})$, while a single one properly contains the hypercenter. Thus two system normalizers are necessary and sufficient.
\end{proof}

\begin{remark}\label{rem:m1}
For $m=1$, the group $G_{1,q}$ has order $2q^3$ and has $q^2$ system normalizers. Proposition~\ref{prop:minimum} gives minimum $2$, while Propositions~\ref{prop:frattini} and~\ref{prop:quotient} give $|\Phi(G_{1,q})|=q$ and $|Z_\infty(G_{1,q}/\Phi(G_{1,q}))|=q$. 
\end{remark}

\begin{proof}[Proof of Theorem~\ref{thm:A}]
Use the group $G_{m,q}$ in (\ref{eq:group}). Its order, metabelianity, and trivial hypercenter follow from Lemma~\ref{lem:affine}. Proposition~\ref{prop:minimum} gives the minimum $m+1$. By Proposition~\ref{prop:frattini}, $\Phi(G_{m,q})=L$ has order $q$, and Proposition~\ref{prop:quotient} shows that the minimum for $G_{m,q}/\Phi(G_{m,q})$ is two.
\end{proof}

\begin{proof}[Proof of Corollary~\ref{cor:unbounded}]
Fix an odd prime power $q$. Given $n\geq1$, take $m=n$. By Theorem~\ref{thm:A}, the group $G_{n,q}$ is metabelian, satisfies $|\Phi(G_{n,q})|=q$, and fewer than $n+1$ system normalizers cannot have intersection $Z_\infty(G_{n,q})=1$. Thus no bound exists even with $q$ fixed. The remaining assertions follow at once.
\end{proof}

\section*{Report of AI use}
The main argument was generated with the assistance of Albilich, a generative artificial-intelligence assistant for mathematical research developed by the authors. All arguments were subsequently fully understood, completely rewritten, and independently verified by the authors.

\section*{Acknowledgements}
Yang was partially supported by a grant from the Simons Foundation, grant no. 918096.

\section*{Disclosure Statement}
The authors report no conflict of interest.

\section*{Data Availability Statement}
No data were used in this research.

\end{document}